\documentclass[11pt]{amsart}
\usepackage{amsmath, a4wide}
\usepackage{amssymb,amsthm,amsfonts}
\usepackage{amsrefs}
\usepackage[mathscr]{euscript}
\usepackage{graphicx}
\usepackage{wrapfig}
\usepackage[usenames,dvipsnames,svgnames,table]{xcolor}

\renewcommand{\mathcal}{\mathscr}

\def\R {\mathbb{R}}

\def\N {\mathbb{N}}

\renewcommand{\epsilon}{\varepsilon}
\newcommand{\eps}{\varepsilon}
\newcommand{\e}{\varepsilon}

\renewcommand{\leq}{\leqslant}
\renewcommand{\le}{\leqslant}
\renewcommand{\geq}{\geqslant}
\renewcommand{\ge}{\geqslant}

\newcommand{\Per}{\mathrm{Per}}

\newtheorem{proposition}{Proposition}[section]
\newtheorem{theorem}[proposition]{Theorem}
\newtheorem{corollary}[proposition]{Corollary}

\theoremstyle{definition}
\newtheorem{definition}[proposition]{Definition}

\newtheorem{remark}[proposition]{Remark}
\numberwithin{equation}{section}

\title{$K$ mean-convex and $K$-outward minimizing sets} 

\author[A. Cesaroni, M. Novaga]{}

\subjclass{ 53C44, 35R11, 49Q20.}

\keywords{Nonlocal curvature flows, mean convexity, nonlocal minimal surfaces, level set flow, minimizing movements.}

 \email{annalisa.cesaroni@unipd.it}
 \email{matteo.novaga@unipi.it}

\begin{document}

\maketitle

\centerline{\scshape Annalisa Cesaroni}
\medskip
{\footnotesize
 \centerline{Dipartimento di Scienze Statistiche}
   \centerline{Universit\`a di Padova}
   \centerline{Via Cesare Battisti 141, 35121 Padova, Italy  }
} 

\medskip

\centerline{\scshape Matteo Novaga}
\medskip
{\footnotesize
 \centerline{Dipartimento di Matematica}
   \centerline{Universit\`a di Pisa}
   \centerline{Largo Bruno Pontecorvo 5, 56127 Pisa, Italy  }
}

\medskip

\begin{abstract}
We consider  the evolution  of sets by  nonlocal mean curvature and
we discuss the preservation along the flow 
of two geometric properties, which are the mean convexity and the outward minimality. 
The main tools in our analysis 
are the level set formulation and the minimizing movement scheme  for the nonlocal flow.
When the initial set is outward minimizing, we also show the convergence of the (time integrated) 
nonlocal perimeters of the discrete evolutions to the nonlocal perimeter of the limit flow.
 \end{abstract} 

\maketitle

 \tableofcontents
\section{Introduction} 

Given an initial set~$E\subset\R^n$, we 
consider its evolution~$E_t$ for~$t>0$ according to
the nonlocal curvature flow
\begin{equation}\label{kflow} 
\partial_t x\cdot \nu=-H^K_{E_t}(x)
\end{equation}  
where $\nu$ is the outer normal at $x\in \partial E_t$.
The quantity~$H_E^K(x)$ is the
$K$-curvature of~$E$ at $x$, which is defined in~\eqref{CURVA} below. 
More precisely, we take a kernel~$K:\R^n\setminus\{0\}\to[0,+\infty)$
such that \begin{equation}\label{ROTAZION2}
\min\{1, |x|\}\, K(x) \in L^1(\R^n)\qquad \text{ and }\qquad K(x)=K(-x)\end{equation}
and we define  the $K$-curvature of a set~$E$ of class $C^{1,1}$,
at~$x\in \partial E$, as
\begin{equation}\label{CURVA} H^K_E(x):=\lim_{\e\searrow0}
\int_{\R^n\setminus B(x,\e)}\Big(
\chi_{\R^n\setminus E}(y)-\chi_E(y)\Big)\,K(x-y)\,dy,\end{equation}
where as usual
$$ \chi_E(y):=\left\{
\begin{matrix}
1 & {\mbox{ if }}y\in E,
\\ 0 & {\mbox{ if }}y\not\in E.
\end{matrix}
\right. $$
For more general sets the $K$-curvature will be understood
 in the viscosity sense (see Definition \ref{viscosity} below) and may be also infinite.

We point out that~\eqref{ROTAZION2} is a very mild integrability assumption,
which fits the requirements in~\cites{imbert, MR3401008}
in order to have existence and uniqueness for the
level set flow associated to~\eqref{kflow}. 
Furthermore, when $K(x)= \frac{1}{|x|^{n+s}}$ for some $s\in (0,1)$, we will denote the 
$K$-curvature of a set~$E$ at a point~$x$ as~$H^s_E(x)$, and we indicate it as
the fractional mean curvature of~$E$ at~$x$.
We also observe that 
the $K$-curvature  is the the first variation of the following nonlocal perimeter functional, see \cites{MR3401008}, 
\begin{equation}\label{per} \Per_K(E):=  \int_{E}\int_{\R^n\setminus E} K(x-y)\,dx\,dy, 
\end{equation} 
and the geometric evolution law~\eqref{kflow} can be interpreted as the $L^2$
gradient flow of this perimeter functional, as shown in~\cite{MR3401008}. 
 
The $K$-curvature flow has been recently studied from different perspectives, mainly in the case of the fractional mean curvature, taking into account  several geometric features. 
In particular  we recall the results about small time existence of a classical solutions \cite{vesa}, 
existence and uniqueness of level set solutions \cites{imbert,MR3401008},
preservation
of convexity \cite{ruf, cinti2}, formation of singularities \cite{cinti1}, classification of symmetric self-shrinkers \cite{cnself}, fattening phenomena \cite{fattening} and stability results for nonlocal curvature flows \cite{av, cdnp}. 

In this paper we are interested in the analysis of  the  flows starting from   $K$-mean convex sets, that is, sets  with positive $K$-curvature,
and from  sets which are one-side minimizers of the nonlocal perimeter functional, the so called $K$-outward minimizing set. This second  property can be interpreted as  the variational analogue 
of the $K$-mean convexity, as we will see in Theorem \ref{euler}. In the case of the fractional curvature, the preservation of the $K$-mean convexity
for smooth sets has been studied in \cite{SAEZ}.  Here we consider more general flows, and also nonsmooth initial data. 
We show that $K$-mean convexity is a too weak condition to be conserved during the evolution, as a consequence we introduce 
the notions of regular $K$-mean convexity and strong $K$-mean convexity (see Definition \ref{kmean}). We introduce the notion of $K$-outward minimality and strong $K$-outward minimality (see Definition \ref{outward}). 
The main results are contained in Theorem \ref{thmconvexity}, about the preservation of regular $K$-mean convexity and strong $K$-mean convexity, 
and Theorem \ref{presmin} about preservation of $K$-outward minimality.   
Our main tools are the level set approach for geometric nonlocal curvature flows, developed in  \cite{imbert,MR3401008}, that we review in Section \ref{sezionelevel}, and the variational scheme,  called minimizing movements or Almgren-Taylor-Wang scheme, introduced  in \cite{ATW,LS} for the classical mean curvature flow,  and  extended to the nonlocal setting in  \cite{MR3401008}. 
 
We conclude by recalling  that, in the local case, there is a vast literature on the analysis of  the mean curvature flow starting from convex sets (see \cite{GH, gigagotoishiisato, MR2238463, soner, belle}) and more generally from
mean-convex sets (see \cite{white, spadaro, dephilaux, chn} and reference therein).
In particular, these geometric properties are preserved by the flow, both in the isotropic and in the anisotropic case, and the singularity formation is 
well understood (see for instance \cite{Hu,HS,HK,HW}).
 
\smallskip
 
 The paper is organized as follows: Section \ref{sezionedefinizioni} contains the definition of $K$-mean convexity and $K$-outward minimality, 
 some examples, and the analysis of the relation between the two notions. 
 Section \ref{sezionelevel} is essentially a review of the level set formulation of nonlocal curvature flows, and  contains  the comparison results between level set flows and classical strict subflows and superflows.  Section \ref{sezionekmean} is devoted to the analysis of the flows starting from $K$-mean convex sets. Section \ref{sezionemin} provides a review of the minimizing movement scheme in the nonlocal setting. Finally, Section \ref{sezionekoutward} contains the analysis of the flows starting from $K$-outward minimizing sets.
 
 \smallskip

\noindent\textbf{Acknowledgments:} The authors are members of the Gruppo Nazionale per l'Analisi Matematica, la Probabilit\`a e le loro Applicazioni (GNAMPA) of the Istituto Nazionale di Alta Matematica (INdAM). M.N. acknowledges partial support by the PRIN 2017 Project
\emph{Variational methods for stationary and evolution problems with singularities and interfaces}.
 
   \section{Main definitions and properties} \label{sezionedefinizioni} 
  In this section we introduce the notions of $K$-mean convexity and $K$-outward minimality, 
  we give some examples and characterizations of these properties, and we analyze their relation.
 
We now recall the definition of constant $K$-mean curvature in the viscosity sense,
for more details we refer to  \cites{imbert, MR3401008} 
and  to~\cite[Section~5]{MR2675483}.
  \begin{definition}\label{viscosity} 
Let $E\subseteq \R^n$ and $x\in\partial E$. Then 
\begin{enumerate}
\item $H^K_E(x)\leq c$ if for all sets $F$ with compact boundary of class $C^{1,1}$ such that $E\subseteq F$ and $x\in \partial F$, there holds $H_F^K(x)\leq c$;
\item $H^K_E(x)\geq c$ if for all sets $F$ with compact boundary of class $C^{1,1}$ such that $E\supseteq F$ and $x\in \partial F$, there holds $H_F^K(x)\geq c$;
\item $H^K_E(x)= c$ if  both $H^K_E(x)\geq c$ and $H^K_E(x)\leq c$.
\end{enumerate}
\end{definition} 

From \eqref{CURVA} it follows that the $K$-mean curvature satisfies the following monotonicity property:
if $E\subseteq F$ and $x\in \partial E\cap\partial F$ is a point where both  $H^K_E(x)$ and $H^K_F(x)$ are defined,
then $H^K_E(x)\geq H^K_F(x)$.
As a consequence, the inequalities in Definition \ref{viscosity} are consistent with the definition of $H^K_E$ in
\eqref{CURVA}.

We observe that the viscosity inequality $H^K_E(x)\leq c$ can be checked only at points $x\in \partial E$ where $E$ satisfies an exterior ball condition, that is, there exists $y_0, r_0$ such that $B(y_0, r_0)\subseteq \R^n\setminus E$, $x\in \partial B(x_0, r_0)$. 
Analogously the viscosity inequality $H^K_E(x)\geq c$ can be checked only  at points $x\in \partial E$ where $E$ satisfies an interior ball condition, that is, there exists $y_0, r_0$ such that $B(y_0, r_0)\subseteq  E$, $x\in \partial B(x_0, r_0)$.
In particular, if $E$ is a closed set with empty interior, then the viscosity inequality $H^K_E(x)\geq k$ is always verified for every $k\in \R$.

We will denote as usual the distance between a point $x$ and a set $E$ as $d(x,E)=\inf_{y\in E}|y-x|$, 
and we define  the signed distance from $E$ as follows
 \[d_E(x)=d(x,\R^n\setminus E)-d(x,E).\]
 We define for $\lambda>0$,
 \begin{equation}\label{ingr} E^\lambda:=\{x\in\R^n\text{ s.t. }d_E(x)\geq -\lambda\}=\{x\in\R^n\text{ s.t. }d(x,E)\leq \lambda\}.\end{equation} 
Observe that   if $E$ is a closed set  then $E=\cap_{\lambda>0}E^\lambda$. 
 
Finally, we  define  the distance between two sets $A, B\subseteq \R^n$,  as follows
 \[d(A, B)=\inf_{a\in \partial A, b\in \partial B} |a-b|.\] 

  \begin{definition}[$K$-mean convexity, regular $K$-mean convexity,  strong $K$-mean convexity] \label{kmean} \ \ 
  \begin{enumerate}
\item A closed  set $E\subseteq \R^n$ is {\bf $K$-mean convex} if $H_E^K(x)\geq 0$ for all $x\in\partial E$.
\item A closed set $E\subseteq \R^n$ is {\bf regularly $K$-mean convex} if 
there exists   $\eta_E>0$ and $c_E> 0$ such that  for all $\lambda\in [0, \eta_E]$ \[  H^K_{E^\lambda}(x) \geq -c_E\lambda \qquad \text{for any $x \in \partial E^\lambda$.}\] 
where $E^0=E$.    
\item A closed set $E\subseteq \R^n$ is {\bf strongly $K$-mean convex} if   there exists $\delta\geq 0$  and  $\xi_E>0$  such  that \[
  H^K_{E^\lambda}(x) \geq \delta\qquad \text{for any $x \in \partial E^\lambda$}\]  for every $\lambda\in [0, \xi_E]$.  

To keep track of the constant $\delta$ we will say in the following that $E\subseteq \R^n$ is   {\bf   strongly  $K$-mean convex set} with associated constant $\delta$. 
\end{enumerate} 
  \end{definition}
\begin{remark}\upshape \label{regularset} 
Note that  if $E$   is strongly $K$-mean convex, then $E$ is also regularly $K$-mean convex.\end{remark}
\begin{remark}[Sets with $C^{1,1}$ boundary] \upshape 
Let $E$ be a compact set with $C^{1,1}$ boundary. 

If  $H^K_E(x)\geq \delta$  $\forall x\in \partial E$,  then for all $\delta'<\delta$,  there exists $\xi_E(\delta')$ such that  \[H_{E^\eta}^K(x)\geq \delta'\qquad \text{ for all $\eta\in [0, \xi_E(\delta')]$ and $x\in \partial E^\eta$}\]  due to the continuity of  $H^K$ with respect to $C^{1,1}$ convergence of sets, see \cite{MR3401008}, and therefore  $E$   is strongly $K$-mean convex with constant $\delta' $.

If $H^K_E(x)\geq 0$  $\forall x\in \partial E$,   and $K(x)=\frac{1}{|x|^{n+s}}$, then  \[\text{$E$   is regularly  $K$-mean convex, }\] due to the result about the variation of  fractional curvature with respect to $C^{1,1}$ diffeomorphisms of sets proved in \cite{cozzi}.  
\end{remark}

\begin{remark}[Convex sets] \upshape 
Let $C$ be a  convex   closed set. Then  \[\text{$C$ is strongly $K$-mean convex with associated constant $0$}\] since it is easy to show that $H^K_C(x)\geq 0$ for every $x\in \partial C$ in the viscosity sense, and moreover $C^\lambda$ are convex sets.   Moreover, if $C$ is compact and 
$\text{ supp } K$  is not compact, then there exists $\delta_C>0$ depending on $K$ and $C$ such that \[\text{$C$ is   strongly $K$-mean convex with associated constant $\delta_C$.}\] Indeed, it is easy to check that if $C\subseteq \R^n$ is a convex set of diameter $R$, then \[H^K_C(x)\geq \int_{\R^n\setminus B(0,R)}K(y)dy:=\delta_C\qquad \text{ for every $x\in \partial C$.}\]   
\end{remark}  

\begin{remark}[Set with positive curvature which is not regularly $K$-mean convex]  \label{counter}
We point out that  if $E$ is a set such that $H^K_E(x)\geq \delta>0$   for  $x\in \partial E$, but $\partial E\not\in C^{1,1}$, then in general it is not true that $E$ is regularly $K$-mean convex. 

 We recall the following example studied in  \cite{fattening}. 
 We consider  the fractional kernel in dimension $2$, that is $K(x)=\frac{1}{|x|^{2+s}}$. 
We define the set $E$ as follows
\[ E:= {\mathcal{G}}_+\cup {\mathcal{G}}_-\subseteq\R^2,\] where ${\mathcal{G}}_+$ is the convex hull of~$B((-1,1),1)$ with the origin,
and~${\mathcal{G}}_-$ the convex hull of~$B((1,-1),1)$ with the origin.

Note that  $\partial E\setminus (0,0)$ is $C^{1,1}$ and in  $(0,0)$ the viscosity  supersolution condition $H^K_E(0,0)\geq \delta$ is true for every $\delta$ since there is no interior ball in $E$ containing $(0,0)$, that is,  there are no regular sets $F$ such that $F\subseteq E$ and $(0,0)\in \partial F$.   It is an easy computation to check, using the radial symmetry of $K$,  that  for all $x\neq 0$, $x\in \partial E$ there holds  \[H^s_E(x)\geq \int_{\R^2\setminus B(0, 1+\sqrt{2})} \frac{1}{|y|^{2+s}}dy= \frac{2\pi}{(1+\sqrt{2})^s} .\]

Let   $Q_r=\{(x_1,x_2)\in \R^2\ \text{ s.t. }x_2\in [-r,r], \ -|x_2|\leq x_1\leq |x_2|\}$.  It has been proven in \cite[Lemma 7.1]{fattening} that there exists a   constant $c>0$ depending on $s$ such that for all $r<c$ there holds 
\[H^s_{E\cup Q_r}(t, r), \ H^s_{E\cup Q_r}(t,- r)\leq -\frac{c}{r^s}\qquad \text{ for all $t\in (-r,r)$}.\] 
Note that  for every point $(t,- r), (t,r)$ with $t\in (-r,r)$ there exists a neighborhood where $\partial (E\cup Q_r)$ is $C^{1,1}$, therefore the previous inequality holds in classical sense.  
Consider now $E^r=\{x\in \R^n\text{ s.t. }d(x, E)\leq r\}$ and note that $(0,r)\in \partial E^r$.   Let  $F$  be a set with boundary $C^{1,1}$ such that $F\subseteq E^r$, $(0,r)\in \partial F$ and such that there exists $\delta<<r$ for which $\partial F\cap B((0,r),\delta)= \partial (E\cup q_r)\cap B((0,r), \delta)$.   Then $ H^s_{F}(0,r) \leq  -\frac{c}{r^s}$.

 If $E$ were regularly $K$-mean convex, there would exist $c_E>0$ such that  $ H^s_{F}(0,r) \geq  -c_E r$  for every $r\in [0, \eta_E]$. Therefore we would get $-\frac{c}{r^s}\geq -\eta_E r$ for every $r\in [0, \eta_E]$, which is not possible. 
We conclude that $E$   is not regularly $K$-mean convex.  \end{remark} 

 Given a measurable set  $E\subseteq \R^n$ and an open set $\Omega\subseteq\R^n$ we let
 \[
 \Per_K(E,\Omega):=\int_E \int_{\Omega\setminus E} K(x-y)\,dx\,dy + \int_{E\cap\Omega} \int_{\R^n\setminus(\Omega\cup E)} K(x-y)\,dx\,dy\,.
 \]
 Notice that, if $E\subset\Omega$ then $\Per_K(E,\Omega)= \Per_K(E)$, in particular $\Per_K(E,\R^n)= \Per_K(E)$ for all sets $E$.
 
  \begin{definition}[$K$-outward minimizing set and strongly $K$-outward minimizing set] \label{outward}  \ \ \\
  Let $\Omega\subseteq\R^n$ be an open set. 
 $E\subseteq \R^n$ is a {\bf $K$-outward minimizing set} in $\Omega$ if  for every $F\subseteq \R^n$ such that $E\subseteq F$ and $F\setminus E\subset\subset \Omega$  there holds \[\ \Per_K(E,\Omega)\leq  \Per_K(F,\Omega).\]
 
  $E\subseteq \R^n$ is   {\bf   strongly  $K$-outward minimizing set} in $\Omega$ if  there exists $\delta>0$ for which for every $F\subseteq \R^n$ such that $E\subseteq F$ and $F\setminus E\subset\subset \Omega$  there holds \[\ \Per_K(E,\Omega)\leq  \Per_K(F,\Omega)-\delta|F\setminus E|.\] 
  
 To keep track of the constant $\delta$ we will say in the following that $E\subseteq \R^n$ is  strongly  $K$-outward minimizing set with associated constant $\delta$. 
 \end{definition}

We now provide some  equivalent characterizations of $K$-outward minimality and strong $K$-outward minimality,  which imply in particular  the stability 
under $L^1$ convergence of $K$-outward minimizing sets. 

\begin{proposition}\label{conv} 
Let $\Omega\subseteq \R^n$ be a domain. The following assertions are equivalent: 
\begin{enumerate} \item   $E$ is a $K$-outward minimizing set in $\Omega$ (resp. strongly $K$-outward minimizing set with associated constant $\delta>0$).
\item  For every  $G\subseteq\R^n$ such that $G\setminus E\subset\subset \Omega$ there holds that 
\begin{equation}\label{inter}
\Per_K(E\cap G,\Omega)\leq \Per_K(G,\Omega), \qquad(\text{resp. }\Per_K(E\cap G,\Omega)\leq \Per_K(G,\Omega)-\delta|G\setminus E|).\end{equation}
\item  For all $A\subseteq  \Omega\setminus E$,  $A\subset\subset \Omega$ there holds that 
\begin{equation}\label{uno} \int_{A}\int_{E} K(x-y)dxdy\leq \int_A\int_{\R^n\setminus (A\cup E)}K(x-y)dxdy\end{equation} 
\[\left(\text{resp. }\int_{A}\int_{E} K(x-y)dxdy\leq \int_A\int_{\R^n\setminus (A\cup E)}K(x-y)dxdy -\delta |A|\right).\]
\end{enumerate}
In particular, if $E_n $ is a sequence of $K-$outward minimizing sets (resp. strongly $K$-outward minimizing sets  with associated constant $\delta$)  in $\Omega$ such that $E_n\to E$ in $L^1(\Omega)$, then $E$ is a $K$-outward minimizing set in $\Omega$ (resp. a strongly $K$-outward minimizing set with associated constant $\delta$). 
\end{proposition} 

\begin{proof}  
We proof the characterization just for $K$-outward minimizers, since the case of strongly $K$-outward minimizers is completely analogous. 
We recall that for all $A, B\subseteq\R^n$, the following submodularity property holds  \begin{equation}
\label{sub}\Per_K(A,\Omega)+\Per_K(B,\Omega)\geq \Per_K(A\cap B,\Omega)+\Per_K(A\cup B,\Omega),\end{equation} 
see e.g. \cite{cn}. 

If \eqref{inter} holds, then it is immediate to check  Definition \ref{outward}: we fix      $F\supseteq E$,  with $F\setminus E\subset\subset \Omega$ and we apply \eqref{inter} to $G=F$. 
  On the other hand, if $E$ is a $K$-outward minimizing set in $\Omega$ and
  $G$ is such that $G\setminus E\subset\subset \Omega$,  letting $F=G\cup E$ and using the submodularity for the first inequality and  Definition \ref{outward} for the second one,  we get
  \[\Per_K(E,\Omega)+\Per_K(G,\Omega)\geq \Per_K(F,\Omega)+\Per_K(G\cap E,\Omega)\geq \Per_K(E,\Omega)+\Per_K(E\cap G,\Omega).\]
  
We now assume that $E$ is a $K$-outward minimizing set in $\Omega$ and we fix $A\subseteq \Omega\setminus E$, with $A\subset\subset\Omega$. Let $F:=E\cup A$, so  that $E\subseteq F$ and $F\setminus E\subset\subset\Omega$. 
By Definition \ref{outward}  we know that 
\[ 
0\le \Per_K(F,\Omega)-\Per_K(E,\Omega) =  \int_{ A}\int_{\R^n\setminus F}K(x-y)dxdy -  \int_{ A}\int_{E}K(x-y)dxdy, 
 \]  which gives \eqref{uno}. 
On the other hand, if we assume that \eqref{uno} holds and fix $F$ such that $E\subseteq F$ and $A:=F\setminus E\subset\subset \Omega$,
then \eqref{uno} gives
\[
\Per_K(F,\Omega)-\Per_K(E,\Omega) =  \int_{ A}\int_{\R^n\setminus F}K(x-y)dxdy -  \int_{ A}\int_{E}K(x-y)dxdy\geq 0,
\]
which implies that  $E$ is $K$-outward minimizing.

Finally, the stability under $L^1$ convergence is a direct consequence of \eqref{inter} and of  the lower semicontinuity of $\Per_K$. Indeed  fix $F$ such that $E\subseteq F$ and $F\setminus E\subset\subset \Omega$. Since $E_n\to E$ in $L^1(\Omega)$, we get that for $n$ sufficiently large $F\setminus E_n\subset\subset \Omega$. Then   by  the fact that $E_n$ are $K$-outward minimizers in $\Omega$,  
$\Per_K(E_n\cap F,\Omega)\leq  \Per_K(F,\Omega) $, and we conclude by lower semicontinuity of $\Per_K(\cdot,\Omega)$ that 
$\Per_K(E\cap F,\Omega)\leq  \Per_K(F,\Omega) $. 
\end{proof} 
 
\begin{remark}[Hyperplanes and   convex sets]\upshape
  Let $\nu\in \R^n$ with $|\nu|=1$ and define the hyperplane $H=\{x\in \R^n\text{ s.t. } x\cdot \nu\geq 0\}$. Then $H$ is a $K$-outward minimizer in every ball $B(0, R)$ for $R>0$, since $H$ is a  local minimizer of $\Per_K$ in every ball $B(0,R)$, see \cite{pagliari}. 

Moreover, every convex set $C$  is a $K$-outward minimizer in every ball $B(0, R)$ for $R>0$. 
Indeed  $C=\cap_{j\in J} H_j$ with $H_j$ hyperplanes. Let $E$ such that $E\setminus C\subset\subset B(0,R)$. Then $E\setminus H_i\subset\subset B(0,R)$ for every $i\in J$ and by minimality of $H_i$ we get
\[
\Per_K(C\cap E,B(0,R))=\Per_K( \cap_j H_j\cap E,B(0,R))\leq \Per_K( \cap_{j\neq i}H_j\cap E,B(0,R)).\] By repeating the same argument  for every $j\in J$, we conclude $\Per_K(C\cap E,B(0,R))\leq \Per_K(E,B(0,R))$.
\end{remark} 

We now analyze the relation between $K$-outward minimality and $K$-mean convexity for compact sets. 
In some sense, (strong) $K$-outward minimality is the variational analogue of (strong) $K$-mean convexity. 
\begin{theorem}\label{euler}\  \ \ \begin{enumerate} 
\item Let $E\subset \subset \Omega$ be a $K$-outward minimizing set  in $\Omega$. Then $H^K_E(x)\geq 0$  for all $x\in \partial E$. If moreover $E$ is  a strongly  $K$-outward minimizing set with associated constant $\delta>0$ then  $H^K_E(x)\geq \delta>0$, for all $x\in \partial E$.  

\item  Let $E\subset\R^n$ be a bounded set, with boundary of class $C^{1,1}$, strongly $K$-mean convex with associated constant $\delta\geq 0$.  Then   there exists an open set $\Omega$, such that $E\subset\subset \Omega$ and $E$ is  a $K$-outward minimizer in $\Omega$ if $\delta=0$, or it is a  strongly $K$-outward minimizer, with associated constant  $\delta$, if $\delta>0$. 
 \end{enumerate} 
\end{theorem}

\begin{proof}
\noindent
\begin{enumerate}\item  
For the case of fractional perimeters, this result has been proved proved in  \cite[Proposition 5.1]{MR2675483}.  Let $\delta\geq 0$. If $E$ is a $K$-outward minimizer, we choose $\delta=0$, if $E$ is a strongly $K$-outward minimizer, we choose $\delta>0$ to be the constant associated to $E$ according to Definition \ref{outward}. We proceed by contradiction and we assume there exists $x_0\in \partial E$, $F\subseteq E$ with $\partial F\in C^{1,1}$, $x_0\in \partial E \cap \partial F$, and $H_F^K(x_0)\leq \delta -2\rho<\delta$ for some $\rho>0$.  Then by continuity of $H^K$ there exists $r>0$ such that $H_F^K(x)\leq \delta-\rho$ for every $x\in\partial E\cap B(x_0,r)$.  We construct a $1$-parameter family $\Phi_\eps$ of $C^{1,1}$ diffeomorphisms, such that $F=\Phi_0(F)\subseteq \Phi_\eps(F)\subset\Omega$ and  $\Phi_\eps(F)\setminus F\subset\subset  B(x_0, r)\subset\subset\Omega$ for every $\eps\in (0, \eps_0)$. Again by continuity there holds   $H^K_{\Phi_\eps(F)}(x)\leq \delta-\rho/2$ for all $x\in \partial  \Phi_\eps(F)\setminus F$. Using the fact that $H^K$ is the first variation of $\Per_K$ with respect to $C^{1,1}$ diffeomorphisms, we get, see    \cite[Proposition 5.2]{MR3401008},
\begin{equation}\label{un}\Per_K(\Phi_\eps(F))=\Per_K(F)+ \int_{ \Phi_\eps(F)\setminus F} H^K_{\Phi_{\eps(x)}(F)}(x)dx\end{equation} where $\eps(x):= \sup\{\lambda\in (0, \eps)\ x\in \Phi_{\lambda}(F)\}$  and
\begin{equation}\label{do}\Per_K(E\cap \Phi_\eps(F)\geq \Per_K(F)+\int_{(E\cap \Phi_\eps(F))\setminus F} H^K_{\Phi_{\eps(x)}(F)}(x)dx.\end{equation}
From \eqref{un}, \eqref{do},  recalling that $H^K_{\Phi_\eps(F)}(x)\leq \delta-\rho/2$ in $\Phi_\eps(F)\setminus E\subseteq \Phi_\eps(F)\setminus F$,  we conclude that \begin{eqnarray*}  \Per_K(E\cap \Phi_\eps(F))&\geq & \Per_K(\Phi_\eps(F))-  \int_{ \Phi_\eps(F)\setminus E}H^K_{\Phi_{\eps(x)}(F)}(x)dx\\ &\geq &  \Per_K(\Phi_\eps(F))
+\left(-\delta+\frac{\rho}{2}\right) |\Phi_\eps(F)\setminus E|\\&> &  \Per_K(\Phi_\eps(F))
-\delta  |\Phi_\eps(F)\setminus E|
. \end{eqnarray*} in contradiction with the fact that   $E$ is a $K$-outward minimizing set in $\Omega$ if $\delta=0$ or a strong $K$-outward minimizing set if $\delta>0$.
\item 
We let   $\Omega:=E^{\xi_E}$, so $E\subset\subset \Omega$ and by \cite[Proposition 5.2]{MR3401008}, for every $F$ with $\Per_K(F)<+\infty$  such that $E\subset F\subset E^{\xi_E}$
 there holds 
 \[\Per_K(F)\geq \Per_K(E)+\int_{F\setminus E} H^K_{\{y\text{ s.t. } d_E(y)\geq  d_E(x)\}}(x)dx\geq \Per_K(E)+\delta|F\setminus E|, \] where the last inequality comes from the fact that    for $x\in F\setminus E$, there holds that $-\xi_E<d_E(x)<0$ and by recalling that $H^K_{E^\lambda}(x)\geq\delta$ for all $\lambda\in [0,\eta']$.  This implies that $E$ is $K$-outward minimizing in $\Omega$  with associated constant $\delta$.
\end{enumerate} 
\end{proof} 

\begin{remark} 
We point out that  $K$-mean convexity does not imply $K$-outward minimality. In particular  if $E$ is a set such that $H^K_E(x)\geq \delta>0$ for all $x\in \partial E$, but $\partial E\not\in C^{1,1}$, then  it is not always true that there exists $\Omega\supset E$ such that $E$ is a $K$-outward minimizing set in   $\Omega$.  We consider the example described in Remark \ref{counter} of a set $E\in \R^2$ which satisfies $H^s_E(x)\geq \frac{2\pi}{(1+\sqrt{2})^2}$ 
for all $x\in \partial E$ and which is not $K$-outward minimizing.  

We recall, see Remark \ref{counter},  that 
\[H^s_{E\cup Q_r}(t, r), H^s_{E\cup Q_r}(t,- r)\leq -\frac{c(n)}{r^s}\qquad \text{ for all $t\in (-r,r)$},\] 
where  $Q_r=\{(x_1,x_2)\in \R^2\ \text{ s.t. }x_2\in [-r,r], \ -|x_2|\leq x_1\leq |x_2|\}$.  
Then, arguing  exactly as in \cite[Proposition 1.8]{fattening}  it is possible to show that $\Per_s(E\cup Q_r)<\Per_s(E)$, 
which implies that $E$ is not a $K$-outward minimizing set. 
  \end{remark} 
 
\section{Level set formulation}\label{sezionelevel} 
 
In this section we recall the level set formulation of the geometric flow \eqref{kflow} in the setting of viscosity solutions for nonlocal equations,
and we collect some results  that will be useful in the sequel.

The viscosity theory for the classical  mean curvature flow is
contained in~\cite{MR1100211,MR1100206}, see also~\cite{MR2238463}
for a comprehensive presentation of the level set approach for classical geometric
flows. The existence and uniqueness of solutions for the fractional curvature flow in~\eqref{kflow} in the
viscosity sense have been investigated  in~\cite{imbert} by
introducing  the level set formulation of the geometric evolution problem~\eqref{kflow} and 
a proper notion of viscosity solution.
The paper   \cite{MR3401008}  is the  main reference where it is introduced 
 a general framework for the analysis via the level set formulation of a 
wide class of local and nonlocal translation-invariant geometric flows.
 
The level set flow associated to~\eqref{kflow} can be defined as follows. 
Given a  closed  set $E\subseteq \R^n$ 
we choose a   Lipschitz continuous function 
$u_E:\R^n\to \R$ such that 
\begin{eqnarray}\nonumber &&
\partial E=\{x\in\R^n  \text{ s.t. } u_E(x)=0\}=\partial\{x\in\R^n\text{ s.t. }  u_E(x)\geq 0\}\\ \label{u}{\mbox{and }} && E=\{x\in\R^n\text{ s.t. } u_E(x)\geq 0\},
\end{eqnarray} e.g. $u_E(x)=d_E(x)$. 
Let also~$u_E(x,t)$ be the viscosity solution of the following nonlocal parabolic problem
\begin{equation}\label{levelset}
\begin{cases}
\partial_t u(x,t)+|Du(x,t)| H^K_{\{y\text{ s.t.} u(y,t)\geq u(x,t)\}}(x)=0,\\
u(x,0)= u_E(x).
\end{cases} 
\end{equation}  For the definition of viscosity solution we refer to  \cite{MR3401008}, see also \cite{imbert}. 
We observe that  the inequality  $H_E^K(x)\leq c$ (resp. $\geq c)$  for $x\in \partial E$ can be shown to be equivalent to 
$H^K_{\{y\text{ s.t.} u_E(y)\geq 0\}}(x)\leq c$ (resp. $\geq c)$ for $x$ with $u_E(x)=0$, in the viscosity sense. 

Due to the comparison principle  proved in full generality  in \cite{MR3401008} the system \eqref{levelset} admits a unique viscosity solution for every initial datum $u_E$ which is    uniformly continuous. Moreover
if $u_E$ is Lipschitz continuous, the solution is still Lipschitz continuous in $x$ with the same Lipschitz constant.

\begin{remark}[Outer and Inner flow]\label{ineout}\upshape 
We define the outer and inner flows defined as follows:
\begin{equation}
\label{outin} E^+(t):= \{x\in\R^n\text{ s.t. }  u_E(x,t)\geq 0\} \qquad
{\mbox{and}} \qquad E^-(t):= \{x\in\R^n\text{ s.t. }  u_E(x,t)>0\}
\end{equation}  where $u_E(x,t)$ is the unique viscosity solution to \eqref{levelset} with initial data $u_E$ as defined in \eqref{u}.
The level set flow of $\partial E$ is given by 
\begin{equation}\label{sigmaet} \Sigma_E(t):=\{x\in\R^n\text{ s.t. }  u_E(x,t)=0\}. \end{equation} 
We observe that   since 
the equation in \eqref{levelset} is geometric, if we replace the initial condition $u_E$ with any function $u_0$ with the same level sets $\{u_0 \geq 0\}$ and $\{ u_0 > 0 \}$, the evolutions $\Sigma_E(t)$,  $E^+(t)$ and $E^-(t)$ remain the same.
For more details, we refer to \cites{imbert, MR3401008}.

Finally we observe that, if $\text{ int }E=\emptyset$, then $u_E(x)\leq 0$ for every $x\in \R^n$, by \eqref{u}. Therefore, by the comparison principle  proved   in \cite{MR3401008} we get that
$u_E(x, t)\leq 0$ for every $t>0$. In particular this implies that 
\begin{equation}\label{empty} \text{ if $E$ has empty interior then  $E^-(t)=\emptyset$ for all $t\ge0$.}
\end{equation}
\end{remark} 
 
Finally we recall some results about comparison between the level set flow and geometric regular subsolutions and supersolutions  to  \eqref{kflow}, which have been proven in  \cite[Appendix]{fattening} (see also \cite{MR3401008}). 

We start with a geometric comparison principle proven in  \cite[Corollary A8]{fattening}.

\begin{proposition} \label{cpgeometric}$\,$ \begin{itemize}\item[i)] 
Let $F\subset E$ two closed sets  in $\R^n$ such that  
$d(F,E)=\delta>0$. Then $F^{+}(t)\subset E^{-}(t)$ for all $t>0$,
and the map  $t\to d(F^{+}(t), E^{-}(t))$ is
nondecreasing. 
\item[ii)] Let $v:\R^n\times[0, T)\to \R$ be a bounded uniformly continuous
viscosity supersolution to \eqref{levelset}, and assume that $
F\subseteq \{x\in\R^n\text{ s.t. } v(x,0)\geq 0\}.$
Then $$F^+(t)\subseteq \{x\in\R^n\text{ s.t. } v(x,t)\geq 0\},\qquad \text{ for all $t\in (0, T)$. }$$

Moreover, if $d(F, \;\{x\in\R^n\text{ s.t. } v(x,0)> 0\})=\delta>0,$ then $$
F^+(t)\subseteq \{x\in\R^n\text{ s.t. } v(x,t)>0\},\qquad \text{ for all $t\in (0, T)$,}$$ and 
$$d\Big(F^{+}(t),   \{x\in\R^n\text{ s.t. } v(x,t)>0\}\Big)\geq \delta .$$ 
\item[iii)]  Let $w:\R^n\times[0, T)\to \R$
be a bounded  uniformly continuous viscosity subsolution to \eqref{levelset},
and assume that $E\supseteq \{x\in\R^n\text{ s.t. } w(x,0)\geq 0\}).$
Then $$E^+(t)\supseteq \{x\in\R^n\text{ s.t. } w(x,t)\geq 0\},\qquad\text{ for all $t\in (0, T)$. }$$

Moreover, if $d(E,\{x\in\R^n\text{ s.t. } w(x,0)\geq 0\})=\delta>0,$
then $$E^{-}(t)\supseteq \{x\in\R^n\text{ s.t. } w(x,t)\geq 0\},\qquad\text{ for all~$t
\in(0,T)$,}$$ and 
$$d\Big(E^{-}(t), \;\{x\in\R^n\text{ s.t. } w(x,t)\geq 0\}\Big)\geq \delta.$$ 
\end{itemize} 
\end{proposition} 
  
We now state a comparison result between the level set flow and
geometric subsolutions or supersolutions to \eqref{kflow}. We omit its proof since it follows exactly 
as in  \cite[Proposition A.10]{fattening}.

\begin{proposition}\label{subgeometrico}
Let $C(t)\subseteq \R^n$ for $t\in [0,T]$, be a continuous family of closed sets with compact boundaries,
and let $E\subseteq\R^n$ be a closed set.


\begin{itemize}\item[i)]
Assume  that  $C(t)$ satisfies a uniform  interior ball condition at every point of its boundary, and that there exists $\delta>0$ such that at every~$
x\in \partial C(t)$  there holds 
\begin{equation}\label{supergeo} 
\partial_t x\cdot \nu(x)+H^K_{C(t)}(x)\geq \delta.\end{equation}
If~$E  \subseteq C(0)$, with $d(E, C(0))=k\geq 0$, then  $E^+(t)\subseteq C(t)$ for all $ t\in [0,T]$, with  $d(E^+(t), C(t))\geq k$.
\item[ii)]   Assume  that  $C(t)$ satisfies a uniform exterior ball condition at every point of its boundary,  and  that there exists $\delta>0$ such that
at every $x\in \partial C(t)$ there holds 
\begin{equation}\label{subgeo} \partial_t x\cdot \nu(x)+H^K_{C(t)}(x)\leq-\delta.
\end{equation}
If  $E \supseteq C(0)$,  then $E^+(t)\supseteq C(t)$ for all $ t\in[0, T]$. \\
If $d(C(0), \{x\in\R^n\text{ s.t. }u_E(x)> 0\})=k>0$, then  $E^-(t)\supseteq C(t)$ for all $ t\in[0,T]$, with  $d(E^-(t), C(t))\geq k$.
\end {itemize} 
\end{proposition}
 
\section{$K$-flow of  $K$-mean-convex sets}\label{sezionekmean} 

In this section we discuss some properties of the $K$-flow \eqref{kflow} starting from a regularly or strongly $K$-mean convex set.  
We first show that  the flow is monotone in the following sense. 

\begin{proposition}\label{chara}  \ \ \ 
\begin{enumerate}\item Let   $E\subseteq \R^n$ be a  strongly $K$-mean convex with associated constant $\delta>0$. 
\\
If $\text{int }E=\emptyset$ then   $E^-(t)=\emptyset$    and $\text{int }E^+(t)=\emptyset$  for every $t\ge 0$,  whereas if $\text{int }E\neq \emptyset$, there holds  
    \begin{equation}\label{concl3new}  E^+(t+s)\subseteq  E^-(t)\text{ with }d(E^+(t+s), E^-(t))\geq \delta s \qquad \text{  for every  $t\geq 0, s\in [0, \xi_E/\delta)$}\end{equation}  where $E^-(0)=\text{int }E$. 
In particular    $E^{+}(t)\setminus E^-(t)$ has empty interior for all $t>0$.
 \item Let $E\subseteq \R^n$ be a   regularly $K$-mean convex. Then there holds 
 \begin{equation}\label{concl1new} E^+(t)\subseteq  E  \text{  and   }  E^+(t+s)\subseteq  E^+(t)\qquad \text{  for every  $t,s\geq 0$.}
\end{equation}
\end{enumerate} 
 \end{proposition} 
 \begin{proof} 
 
 \noindent
 \begin{enumerate}\item 
Let $\delta>0$ and $\xi_E$ be the constants associated to $E$, according to Definition \ref{kmean}. Let $\xi\leq \xi_E$. For $0<h< \min( \delta, \xi)$ and $s\in [0, 1]$  \[C(s):=E^{\xi-h s}.\]  
We observe that  $C(s)$ is a supersolution to \eqref{kflow}, in the sense that it
satisfies 
\[ \partial_s x\cdot \nu+H^K_{C(s)}(x)=- h+H^K_{C(s)}(x)>0.\]  
Since $E\subseteq E^{\xi}=C(0)$, by Proposition \ref{subgeometrico},
we get that for all $s\in (0, 1]$ there hold, for every $\xi\leq \xi_E$, 
\[
E^{+}(s)\subseteq  C(s)=E^{\xi-hs} \subseteq E^\xi\ \ \quad\text{and }  \quad 
d\left(E^+(s),E^{\xi-hs}\right)\geq d(E, E^{\xi})=\xi.\]
This implies   that  for all $s\in [0,1]$ \[E^{+}(s)\subseteq \cap_{0<\xi\leq \xi_E} E^\xi=E.\]   \\
 Therefore, if $\text{int}(E)=\emptyset$, we conclude that $\text{int }E^+(t)=\emptyset$ and we recall that $E^-(t)=\emptyset$ for all $t>0$ by  \eqref{empty}.\\
 Assume now  that $E$ has nonempty interior.  Arguing as above we define  $C(s)=E^{\xi_E-\delta s}$ and we get that $C(s)$ is a supersolution to \eqref{kflow} for every $s\in [0, \xi_E/\delta)$. Therefore as above, by Proposition \ref{subgeometrico}, we get that $E^+(s)\subseteq E^{\xi_E-\delta s}$ for every $s\in [0, \xi_E/\delta)$ and $d(E^+(s), E^{\xi_E-\delta s})\geq d(E, E^{\xi_E})=\xi_E$. 

Let $x\in \partial E^+(s)$.  Then $d(x, \partial E^{\xi_E-\delta s})\geq d(E^+(s), E^{\xi_E-\delta s})\geq \xi_E$. Therefore, for every $y\in \partial E$ we get 
 \begin{eqnarray*}
 \xi_E&\leq& d(x, \partial E^{\xi_E-\delta s})=\min_{z\in \partial E^{\xi_E-\delta s}} |x-z|\leq |x-y|+\min_{z\in \partial E^{\xi-\delta s}} |y-z|
 \\
 &=& |x-y|+\xi_E-\delta s,
 \end{eqnarray*}
 which in turn gives that for all $x\in \partial E^+(s)$ with $s\in [0,\xi_E/\delta)$  and all $y\in \partial E$ there holds    
\[|x-y|\geq \delta s.\] This implies that for all $s\in [0,\xi_E/\delta)$ 
 \begin{equation}\label{prova1} d\left(E^+(s), E\right)\geq \delta s>0.\end{equation} 
In particular it follows that $E^+(s)\subseteq \text{int}(E)$.

By the Comparison Principle in
Corollary \ref{cpgeometric}, we get  that
\[E^{+}(t+s)\subseteq E^- (t) \quad 
  {\mbox{ for all }} 
t\geq 0, \ s\in (0, \xi_E/\delta),\quad \text{with }d(E^{+}(t+s), E^- (t) )\geq \delta s.\]
Finally  we recall the following lower semicontinuity result for the outer evolution proved in \cite[Proposition A.12]{fattening}: 
$\liminf_{\eta\to 0}|E^+(t+\eta)|\geq |\text{int }E^+(t)|$.\\ 
Then,  since  $E^{+}(t+s)\subseteq E^-(t)$ for $s\in (0,\xi_E/\delta)$, we get  
\[|\text{int }(E^+(t)\setminus { E^-(t)})|\leq \limsup_{s\to 0^+} |\text{int }E^{+}(t)|-|E^{+}(t+s)|\leq 0\] which gives the conclusion.  

\item 
 Now we consider the  case of a regularly $K$-mean convex set $E$. Let fix $\lambda\leq \eta_E$ and $T<\frac{1}{c_E}$ and define the flow  $C(t)=E^{c_E \lambda t}$ for $t\in [0, T]$. Note that since $c_E \lambda t\leq \eta_E$, there holds that $H^K_{C(t)}(x)\geq -c_E^2 \lambda t\geq  -c_E^2 \lambda T>- c_E\lambda $ for all $t\in [0, T]$, which implies that $C(t)$ is a strict supersolution to \eqref{kflow}. Therefore by Proposition \ref{subgeometrico}, we get that 
 $$
 E^+(t)\subseteq E^{  c_E\lambda  t}\quad\text{ for all $0\leq t\leq T<\frac{1}{c_E}$  and every $\lambda\in (0, \eta_E]$. }
 $$
 This implies that for $t\in \left[0, \frac{1}{c_E}\right)$, $E^+(t)\subseteq \cap_{\lambda\in (0,\eta_E] }E^{c_E\lambda t}=E$, since $E$ is closed.  Then by the Comparison Principle in
Corollary \ref{cpgeometric}, we get  that 
\[E^{+}(t+s)\subseteq E^+ (t) \quad \
  {\mbox{ for all }} 
t\geq 0, \ s\in  \left[0, \frac{1}{c_E}\right).\]
\end{enumerate}
 \end{proof}

 \begin{remark}\upshape  \label{fattrem} Observe that if $E$ is $K$-mean convex and $H^K_E(x)\geq \delta>0$ for all $x\in \partial E$ in viscosity sense, but $E$ is not regularly or strongly $K$-mean convex, then in general it is not true that $E^{+}(t)\subseteq E$ for $t>0$ and moreover it is not true that the flow does not develop fattening.   Fattening  phenomenon  is related to non-uniqueness of the geometric flow; for an analysis of this phenomenon, mainly in dimension $2$,   for geometric equations as \eqref{kflow}, we refer to \cite{fattening}. 
 
As an  example we consider the set  $E$ described in Remark \ref{counter}.  In \cite[Thm 1.10]{fattening} it is proved that there exists $t>0$ and $c>0$ such that $E^-(t)\subseteq B(0, r(\tau))\subseteq E^+(\tau)$ for all $\tau\in [0, t)$, where $r(\tau)=c(n) \tau^{1/(1+s)}$, so implying that \eqref{concl1new} cannot hold. \end{remark} 

Moreover, we show that  monotonicity of the flow implies $K$-mean convexity.
 \begin {proposition}\label{pos2} Let $E$ be a closed  set. 
Assume that  there  exists $h>0$ such that \begin{equation}\label{concl1} E^+(t)\subseteq  E \qquad\text{  for every $0\leq t\leq h$, }\end{equation}
then $H^K_E(x)\geq 0$ in viscosity sense for every $x\in \partial E$. 

If moreover there exists $\delta>0$ such that \[ E^+(t)\subseteq  E  \text{ with }d(E, E^+(t))\geq \delta t\qquad\text{ for every  $0\leq t\leq h$,}\] then $H^K_E(x)\geq \delta$ in viscosity sense for every $x\in \partial E$.
 \end{proposition} \begin{proof} We prove  directly the second statement, since the first can be proved in a similar way, just putting $\delta=0$. Assume  that  it is not true that $H^K_E(x)\geq \delta$ in viscosity sense for every $x\in \partial E$. Therefore there exists $x\in \partial E$  and a set $F$ with $C^{1,1}$ boundary such that $F\subseteq E$, $x\in \partial F\cap\partial E$ and $H^K_F(x)\leq \delta-4\rho<\delta$. By continuity of the curvature on regular sets (see \cite{MR3401008}), there exists $r>0$ such that for all $y\in \partial F\cap B(x,4r)$, there holds $H_F^K(y)\leq \delta-3\rho$.

 Now we construct a strict subsolution $C(t)$ to \eqref{kflow} with $C(0)=F$ as follows. Let  $c= \max_{y\in\partial F} H^K_F(y)\geq 0$ and 
 let $\psi_r, \phi_r:\R^n\to [0,1]$ be two  smooth functions such that $\psi_r(y)=1$ for $y\in B(x, r)$ and $\psi_r(y)=0$ for $y\in \R^n\setminus B(x,2r)$,  and on the other hand $\phi_r(y)=0$ for $y\in B(x,3r)$ and $\phi_r(y)=1$ for $y\in\R^n\setminus B(x,4r)$. We construct a family of regular sets as follows: $C(0)=F$ and $C(t)$ is the set whose boundary is 
 \[ \partial C(t)=\{y+(-\delta+\rho) t \psi_r(y)  \nu_{\partial F}(y)-(c+2\rho)t \phi_r(y) \nu_{\partial F}(y)\ \text{ for  all }y\in  \partial F\}\] where $\nu_F(y)$ is the outer normal of $F$ at $x\in\partial F$.  For $t>0$ sufficiently small, $C(t)$ is of class $C^{1,1}$ and moreover,  by continuity of the curvature on regular sets, \begin{equation}\label{ct} H_{C(t)}^K(y)\leq \delta-2\rho\quad \text{for $y\in \partial C(t)\cap B(x,4r)$}\qquad \text{and }\quad c+\rho\geq \max_{y\in \partial C(t)} H^K_{C(t)}(y).\end{equation} Finally, observe that  at every 
 $y\in \partial C(t)$ there holds 
 \[\partial_t y\cdot \nu(y) =(-\delta+\rho) \psi_r(y)-(c+2\rho)\phi_r(y)\leq -H_{C(t)}^K(y)-\rho \] where the last inequality is obtained by recalling the definition of $\phi_r, \psi_r$ and \eqref{ct}. 
We conclude by Proposition \ref{subgeometrico} that, since $C(0)=F\subseteq E$, then $C(t)\subseteq E^+(t)$ for all $t>0$ sufficiently small. 
 
Note that $d(x+(-\delta+\rho)t \nu_F(x),x)=(\delta-\rho)t$  and then $d(C(t), E)\leq (\delta-\rho)t<\delta t$, in contradiction with the fact that $d(E^+(t), E)\geq \delta t$ and $C(t)\subseteq E^+(t)\subseteq E$. 
 
\end{proof} 

\begin{remark}\upshape \label{entro} 
Note that, arguing exactly as in the proof of Proposition \ref{pos2}, we may prove the following result: if $E$ is a closed set such that there exist $\delta>0$ and  $h>0$ for which 
\[\sup_{x\in E^+(t)} d(x, E)\leq \delta t\qquad \forall t\leq h, \]
then \[H^K_{E}(x)\geq -\delta\qquad\text{ in viscosity sense for all $x\in \partial E$}. \]

Indeed we argue by contradiction and we choose $F$ as in the proof of Proposition \ref{pos2},  with $C^{1,1}$ boundary such that $F\subseteq E$, $x\in \partial F\cap \partial E$ and $H^K_F(y)\leq -\delta-2\rho$ for all $y\in \partial F\cap B(x,r)$.   We construct now a strict  subsolution  to \eqref{kflow} as 
 \[\partial C(t)=\{y+(\delta+\rho) t \psi_r(y)  \nu_{\partial F}(y)-(c+2\rho)t \phi_r(y) \nu_{\partial F}(y)\ \text{ for  all }y\in  \partial F\}\] where $c= \max_{y\in\partial F} H^K_F(y)\geq 0$ (since $F$ is compact). Therefore by comparison $C(t)\subseteq E^+(t) $ and $\sup_{x\in C(t)} d(x, E)\geq (\delta+\rho) t$, which gives a contradiction. 
\end{remark} 

We collect the previous results about flows of regularly and strongly   $K$-mean convex sets.
 \begin{theorem}\label{thmconvexity}\ \   \begin{enumerate}\item  Let $E$ be a strongly $K$-mean convex set with associated  constant $\delta\geq 0$. Then for all $\eta\in [0, \xi_E)$ the outer flow $(E^\eta)^+(t)$ is monotone according to  \eqref{concl3new} if $\delta>0$, or to   \eqref{concl1new}  if $\delta=0$ and moreover  there holds 
\[H^K_{(E^\eta)^+(t)}(x)\geq \delta\qquad\text{  for all $t\geq 0$.  }\] 
\item  Let $E$ be a  regularly  $K$-mean convex set. Then  the outer flow $E^+(t)$ is monotone according to \eqref{concl1new}  and  there holds 
\[H^K_{E^+(t)}(x)\geq 0\qquad \forall t\geq 0. \] 
\end{enumerate} 
\end{theorem}
\begin{proof} 
\noindent
\begin{enumerate} \item Note that  by definition if $K$ is strongly $K$-mean convex with associated constant $\delta\geq 0$, then also $E^\eta$, for any $\eta\in (0, \xi_E)$, is  strongly $K$-mean convex with associated constant $\delta\geq 0$, and $\xi_{E^\eta}=\xi_E-\eta$. Therefore, we may apply Proposition \ref{chara}   to every $E^\eta$ and deduce that    if $\delta=$ then \eqref{concl1new}  holds for $(E^\eta)^+(t)$ for every $t\geq 0$ and if $\delta>0$ then \eqref{concl3new} holds   for  $s\in \left[0, \frac{\xi_E-\eta}{\delta}\right]$ and for every $t\geq 0$.  Now, by Proposition \ref{pos2}, we get that \[H^K_{(E^\eta)^+(t)}(x)\geq \delta\qquad\text{  for all $t\geq 0$.  }\]

 \item The fact that $H^K_{E^+(t)}(x)\geq 0$ is a consequence of \eqref{concl3new} and Proposition \ref{pos2}. 
%
 \end{enumerate} 
 
 \end{proof} 

 \section{Minimizing movements}\label{sezionemin}
 
 We now recall  the variational scheme, sometimes called minimizing movements, introduced  in \cite{ATW} for the classical mean curvature flow,
 and later extended to the nonlocal setting in  \cite{MR3401008}. 
 
Given a nonempty set $ E\subseteq \R^n$ with compact boundary and a time step $h>0$, if $E$ is bounded we define the  set $T_h(E)$ as  a solution of the minimization problem
\begin{equation}\label{min}\min_{F\subseteq\R^n} \Per_K(F)-\frac{1}{h} \int_{F} d_{E }(x)dx. 
\end{equation} If  $E$ is unbounded then we define  $T_h(E):=\R^n\setminus T_h(\R^n\setminus E)$.   
We also  let $T_h(\emptyset):=\emptyset$.

We iterate  the scheme to obtain  $T^{(k)}_h(E)=T_h(T_h^{(k-1)}(E))$,  where we put  $T_h^{(1)}(E)=T_h(E)$,  and we define the   following piecewise constant flows as follows
\begin{equation}\label{pol}E_h (t) =T_h^{(k)}(E)\qquad \text{ for } t\in [kh, (k+1)h).\end{equation}       

In the sequel we will identify a minimizer $T_h(E)$, and a time discrete flow $E_h (t)$,
with the representative given by the set of Lebesgue points of the characteristic function.

We recall from \cite{MR3401008} some results about this scheme.

\begin{theorem}\label{proprieta}\ \ \  
\begin{enumerate}
\item For any   set $E$, the minimization problem \eqref{min} admits a maximal $T^{+}_h(E)$  and a minimal solution $T^{-}_h(E)$ (with respect to inclusion).  We will denote the flow obtained in \eqref{pol} by interpolating  respectively the minimal and the maximal solution as respectively $E_h^-(t) $ and  $E_h^+(t) $. Every flow constructed as in \eqref{pol} satisfies $E_h^{-}(t)\subseteq E_h(t)\subseteq E^+(t)$.
\item  If $E\subseteq F$, then $T_h^\pm(E)\subseteq T_h^\pm(F)$. Moreover if $d(E, F)\geq r$, then $d(T_h(E), T_h(F))\geq r$.
\item  There exists a  constant $C>1$ depending only on  the dimension, such that for every fixed $R>0$ and   every $h>0$ such that \[R-h \min_{x\in \partial B(0,CR)}H^K_{B(0,CR)}(x)) >0,\] there holds \[T^{\pm}_h(B(0,R))\subseteq B\left(0, R-h \min_{x\in \partial B(0,CR)}H^K_{B(0,CR)}(x)\right).\] 
\item For every $R_0>0$, $\sigma>1$ there exists $h_0>0$ depending on $R_0, \sigma, C$ such that if $h\leq h_0$, then there holds for any $R\geq R_0$, and $h\leq h_0$, 
 \[ B\left(0, R-h \max_{x\in \partial B(0,R/\sigma)}H^K_{B(0,R/\sigma)}(x)\right)\subseteq T^{\pm}_h(B(0,R)).\]  
\item Let  $E\subset F$ be a nonempty bounded set with $r=d(E,F)>0$. Then there exists $h_0>0$ depending on $r$ and the dimension such that  for all $h\leq h_0$, there holds that $T^{\pm}_h(E)\subseteq F$ and moreover \[d(T^{+}_h(E), F)\geq r-h \max_{x\in \partial B(0,r/2)}H^K_{B(0,r/2)}(x)>0.\]  \end{enumerate} 
\end{theorem}

\begin{proof} For the proof of items (1)-(4) we refer to Proposition 7.1, Lemma 7.2, Lemma 7.4, Lemma 7.5, Lemma 7.6, Lemma 7.10 in \cite{MR3401008}.  

We now show item (5). 
We fix $x\in \partial F$ and observe that by assumption, for every $r'<r$, $E\subset \R^n\setminus B(x, r')$ and then by monotonicity 
\begin{equation}\label{monow} T^\pm(E)\subseteq T^\pm(\R^n\setminus B(x, r'))=\R^n\setminus T^\mp (B(x, r')).\end{equation}    Now, we apply item (4), choosing $R_0=r/2$ and $\sigma=2$: there exists $h_0$ depending on $r$ such that for all $r'>r/2$, and $h\leq h_0$, 
\[ B\left(x, r'-h \max_{y\in \partial B(0,r'/2)}H^K_{B(0,r'/2)}(y)\right)\subseteq T^{\pm}_h(B(x,r')). \]
Substituting in \eqref{monow} we get  for all $x\in \partial F$ 
\[ T^{\pm}_h(E)\subseteq  \R^n\setminus B\left(x, r'-h \max_{y\in\partial B(0, r'/2)} H^K_{B(0,r'/2)}(y)\right)\qquad \text{ for all }r'\in (r/2,r), \ h\leq h_0.\]
 This implies  for all  $h\leq h_0$,  that either $T^{\pm}_h(E)\subseteq F$  or
  $T^{\pm}_h(E)\subseteq \R^n\setminus F$, and in both cases   \begin{equation}\label{bordi}d(F, T^{\pm}_h(E))\geq  r-h \max_{y\in\partial B(0, r/2)} H^K_{B(0,r/2)}(y)>0.\end{equation}  

Finally, we observe that  necessarily $T^{\pm}_h(E)\subseteq F$. Assume by contradiction that $T^{\pm}_h(E)\subseteq \R^n\setminus F$. Then, recalling that $E\subset F$ with $d(E,F)=r$, from \eqref{bordi} we would get that      
$d_E(x)\leq -2r+h \max_{y\in\partial B(0, r/2)} H^K_{B(0,r/2)}(y)<0$  for every $x\in  T^{+}_h(E)$. So, it would be possible, just by translating $T^{+}_h(E)$ to construct a competitor with strictly  less energy, and so to prove that $T_h^{+}(E)$ couldn't be a solution to the minimization problem \eqref{min}. 
\end{proof} 
Finally we recall the convergence of the scheme to $K$-mean curvature flow, as proved in \cite[Proposition 7.12 and Theorem 7.16]{MR3401008}.

\begin{theorem}\label{convschema} 
Let $u_0$ be a  Lipschitz continuous function. We define 
\[ T_h u_0(x):=\sup\left\{\lambda \text{ s.t. }x\in T_h \left(\{x \text{ s.t. } u_0(x)>\lambda\} \right) \right\},\]
and iteratively   for $k\in \N$, 
\begin{equation}\label{thu} T^{(k)}_h u_0(x):= T_h (T^{(k-1)}_h u_0(x)).\end{equation} 
Let \[u_h(x,t):=T^{[t/h]}_h u_0(x) \] then there holds 
\[\begin{cases} 
T_h^-\left(\{x\text{ s.t. } u_h(x, (k-1)h)>\lambda\}\right)&=\ \{x\text{ s.t. } u_h(x, k h)>\lambda\}
\\ 
T_h^+\left(\{x\text{ s.t. } u_h(x, (k-1)h)\geq \lambda\}\right)&=\ \{x\text{ s.t. } u_h(x, kh)\geq \lambda\} \end{cases}\]
where the second equality holds up to a negligible set,
and moreover 
\[u_h(x,t)\to u(x,t)\qquad \text{ as $h\to 0$,  locally uniformly in $\R^n\times [0, +\infty)$,} \]
where $u(x,t)$ is the unique solution to \eqref{levelset} with initial datum $u_0$. 
\end{theorem}

 \section{$K$-flow of $K$-outward minimizing sets} \label{sezionekoutward} 

In this section we show that the level set flow preserves the $K$-outward minimality. In the case of the classical mean curvature flow, we refer to \cite{dephilaux}  for an analysis of outward minimizing sets. In particular in that paper it is shown that these sets provide a   class of initial data for which the minimizing movement scheme converges to the level set flow. For the generalization of this result to the anisotropic case and cristalline case, we refer to \cite{chn}. 

First of all  we show that the  the minimizing movement scheme \eqref{pol} starting from a $K$-outward minimizer is monotone (see  \cite[Lemma 2.7]{dephilaux} for the case of the classical perimeter, and \cite[Lemma 2.3]{chn} for the anisotropic perimeter).

\begin{proposition}\label{prop1} 
Let $\Omega$ be an  open set and let  $E$ be a nonempty bounded set with $E\subset\subset \Omega$. 
If $E$ is a $K$-outward minimizing  set  in $\Omega$, then there exists $h_0$ depending on $r=d(E,\Omega) >0$ such that for all $h\leq h_0$,  every piecewise constant flow  $E_h (t) =T_h^{(k)}(E)$, for $t\in [kh, (k+1)h)$  defined in \eqref{pol}   satisfies
\[E_h(t)\subseteq \overline{E_h(s)}  \quad  \text{ and }\quad \Per_K(E_h(t))\leq \Per_K( E_h(s)) \qquad \forall t\geq s\geq 0\] where $E_h(0)=E$.   
Moreover $E_h(t)$ is a $K$-outward minimizing set in $\Omega$, so that $H^K_{E_h(t)}(x)\geq 0$ in the viscosity sense at every $x\in\partial E_h(t)$. 
 \end{proposition} 
 
\begin{proof}  
First of all we observe that by Theorem \ref{proprieta},   since $E\subset\subset \Omega$, then there exists $h_0$ such that  $T_h^+ (E)\subset \subset \Omega $  for all $h\leq h_0$.  
Now we proceed by induction on $k\geq 0$ and to avoid long notation we will denote  $E_k:=T_h^{(k)}(E)$.
 Since $T_h(E_k)$ is a minimizer of \eqref{min}, choosing $E_k\cap T_h(E_k)$ as a competitor   we get
\begin{eqnarray*}\nonumber  \Per_K(T_h(E_k))- \Per_K(E_k\cap T_h(E_k))&\leq & \frac{1}{h} \int_{T_h(E_k)} d_{E_k}(x)dx-\frac{1}{h} \int_{E_k\cap T_h(E_k)} d_{E_k}(x)dx\\ 
&=& \frac{1}{h} \int_{T_h(E_k)\setminus E_k}d_{E_k}(x)dx\leq 0\label{servepoi}\end{eqnarray*} since  $d_{E_k}\leq 0$ on $\R^n\setminus E_k$.
Since $E_k$ is a $K$-outward minimizer, we get that 
\[\Per_K(E_k\cap T_h(E_k))\leq \Per_K(T_h(E_k))
\quad \text{and}\quad \frac{1}{h} \int_{T_h(E_k)\setminus E_k}d_{E_k}(x)dx=0,
\]
which implies that $T_h(E_k)\subseteq \overline{ E_k}$, up to a negligible set,  recalling that $d_{E_k}<0$ on $\R^n\setminus \overline{E_k}$. 
Now,   using $E_k$ as a competitor, we observe that, since $T_h(E_k)\subseteq \overline{E_k}$ and $d_{E_k}\geq 0$ in $\overline{E_k}$, 
\[\Per_K(T_h(E_k)) \leq  \Per_K(E_k)-\frac{1}{h} \int_{E_k } d_{E_k}(x)dx+\frac{1}{h} \int_{T_h(E_k)} d_{E_k}(x)dx\leq \Per_K(E_k).\]
Let $G\supset T_h(E_k)$  such that $G\setminus T_h(E_k)\subset\subset \Omega$. Our aim is to prove that $\Per_K(T_h(E_k))\leq \Per_K(G)$. Using the minimality of $T_h(E_k)$ and $G\cap E_k$ as competitor we get, recalling that $T_h(E_k)\subseteq \overline{E_k}\cap G$ and that $d_{E_k}=0$ on $\overline{E_k}\setminus E_k$, 
\[\Per_K(T_h(E_k))\leq \Per_K(G\cap E_k)-\frac{1}{h}\int_{G\cap E_k}d_{E_k}(x)dx+\frac{1}{h}\int_{T_h(E_k)}d_{E_k}(x)dx\leq \Per_K(G\cap E_k).\]
We conclude recalling that $E_k$ is a $K$-outward minimizer so that $ \Per_K(G\cap E_k)\leq \Per_K(G)$. 
\end{proof}

\begin{proposition}\label{prop2}
Under the same assumptions of Proposition \ref{prop1}, if $E$ is also  strongly $K$-outward minimizing  set  in $\Omega$ with constant $\delta>0$,   for all $h\leq \min \left(h_0, \frac{d(E, \Omega)}{\delta}\right)$ we have
  \begin{itemize} \item if  $E$ has  empty interior, then $T_h(E)=\emptyset$; 
\item if $E$ has nonempty interior, then 
  the discrete flow $E_h (t)$  satisfies
\[ d(E_h(t), E_h(t+h))\geq \delta h \quad \text{ and }\quad H^K_{E_h(t)}(x)\geq \delta
\quad \text{ for all $t\ge 0$ and $x\in \partial E_h(t)$.}\]\end{itemize} 
\end{proposition}

\begin{proof}
Observe that, by the definition of the piecewise constant flow $E_h(t)$,  it is sufficient to prove the second statement for $E_k:=T_h^{(k)}(E)$ for every $k\geq 1$. 
We start considering the case $k=1$. In this case $E_1=T_h(E)$. By Proposition \ref{prop1}, we know that $E_1\subseteq \overline{E}$. 
We fix $z\in \R^n$ with $|z|< h\delta $ and observe that  $E_1+z\subseteq \overline{E}+z\subset \Omega$ since $h\delta\leq d(E, \Omega)$. Now $E_1+z$ is a solution to the minimization problem
\[\min_F \left(\Per_K(F)-\frac{1}{h}\int_{F} d_E(x-z)dx\right). \] We choose $E\cap (E_1+z) $ as a competitor and we get 
\[\Per_K(E_1+z)-\frac{1}{h} \int_{E_1+z} d_{E}(x-z)dx\leq  \Per_K(E\cap (E_1+z))-\frac{1}{h} \int_{E\cap (E_1+z)} d_{E}(x-z)dx.\]
Since $E$ is a strongly $K$-outward minimizer we get
\[ \Per_K(E\cap (E_1+z))\leq  \Per_K(E_1+z)-\delta |(E_1+z)\setminus E|. \]
Substituting in the previous inequality we get 
\[ \delta |(E_1+z)\setminus E|\leq   \frac{1}{h} \int_{ (E_1+z)\setminus E} d_{E}(x-z)dx.\]
Finally for $x\not\in E$, by definition \[d_E(x-z)=d(x-z, \R^n\setminus E)-d(x-z, E)\leq d(x-z, \R^n\setminus E)\leq d(x-z, x)=|z|.\] 
Therefore in the previous inequality we get
\[\delta |(E_1+z)\setminus E|\leq   \frac{1}{h} |z| |(E_1+z)\setminus E|<\delta |(E_1+z)\setminus E|\] 
which implies that $|(E_1+z)\setminus E|=0$  for every $z$ with $|z|<\delta h$, that is
\[E_1+B(0, \delta h)\subseteq E.\] Note  that  if $\text{int }E =\emptyset$,  then by the previous inclusion  we get that necessarily $E_1=\emptyset$. \\
If $\text{int }E \not=\emptyset$, we have that
\[E_1\subseteq E\qquad \text{and} \qquad d(E_1, E)\geq h\delta.
\]    
By Theorem \ref{proprieta},  we then get 
\[E_2=T_h(E_1)\subseteq T_h(E)=E_1\qquad \text{and }d(E_2, E_1)\geq h\delta.\]
So by iteration we obtain \[E_k\subseteq E_{k-1} \qquad \text{and }\qquad d(E_k, E_{k-1})\geq  h\delta .\]

Finally we fix $k\geq 1$ and we claim that for   any $\lambda\in (0,1)$, there holds 
\[H^K_{E_k}(x)\geq \delta(1-\lambda) \text{ in viscosity sense, for all $x\in \partial E_k$}.\]  So, sending $\lambda\to 0$ we get the statement. 

The minimality of $E_k=T_h(E_{k-1})$ and the submodularity of the perimeter \eqref{sub} give  that for all $G$
\begin{eqnarray}\nonumber  \Per_K(G\cap E_k)&\leq & \Per_K(G)+\Per_K( E_{k}) -\Per_K(E_k\cup G)\\ &\leq & \Per_K(G)-\frac{1}{h}\int_{G\setminus E_k}d_{E_{k-1}}(x)dx.
\label{g1}
\end{eqnarray} 

We proceed as in the proof of Theorem \ref{euler}, item (1). We fix $\lambda\in (0,1)$ and we assume by contradiction that there  exists  $F\subseteq E_k$ with $\partial F\in C^{1,1}$, $x_0\in \partial E_k \cap \partial F$, such that $H_F^K(x_0)\leq \delta(1-\lambda) -2\rho$ for some $\rho>0$ small.  Then by continuity of $H^K$ there exists $r_0>0$ such that $H_F^K(x)\leq \delta(1-\lambda)-\rho$ for every $x\in\partial F\cap B(x_0,r_0)$.
We fix \begin{equation}\label{erre} r<\min\left(r_0,\frac{ h\delta\lambda}{2}\right),\end{equation}  so that $B(x_0, r)\subset\subset E_{k-1}$ (since $d(E_k, E_{k-1})\geq \delta h$) and we construct  a $1$-parameter family $\Phi_\eps$ of $C^{1,1}$ diffeomorphisms, such that  $F=\Phi_0(F)\subseteq \Phi_\eps(F)\subset E$, $|\Phi_\eps(F)\setminus E_k|>0$  and  $\Phi_\eps(F)\setminus E_k \subseteq \Phi_\eps(F)\setminus F\subset\subset  B(x_0, r)\subset\subset E_{k-1}$ for every $\eps\in (0, \eps_0)$. Again by continuity there holds   $H^K_{\Phi_\eps(F)}(x)\leq \delta(1-\lambda)-\rho/2$ for all $x\in \partial  \Phi_\eps(F)\setminus F$. Using the fact that $H^K$ is the first variation of $\Per_K$ with respect to $C^{1,1}$ diffeomorphisms,  as in the proof of Theorem \ref{euler} (see \eqref{un}, \eqref{do}), we get 
\[\Per_K(E_{k}\cap \Phi_\eps(F))\geq  \Per_K(\Phi_\eps(F))
+\left(-\delta(1-\lambda)+\frac{\rho}{2}\right) |\Phi_\eps(F)\setminus E_k|.\] 
 By the previous inequality and \eqref{g1} applied to $G=\Phi_\eps(F)$ we get 
 \[ \Per_K(\Phi_\eps(F))-\frac{1}{h}\int_{\Phi_\eps(F)\setminus E_k}d_{E_{k-1}}(x)dx\geq  \Per_K(\Phi_\eps(F))
+\left(-\delta(1-\lambda)+\frac{\rho}{2}\right) |\Phi_\eps(F)\setminus E_k|
 \] from which we deduce
\[  \frac{1}{h}\int_{\Phi_\eps(F)\setminus E_k}d_{E_{k-1}}(x)dx\leq \left(\delta(1-\lambda)-\frac{\rho}{2}\right) |\Phi_\eps(F)\setminus E_k|.\]
Observe that $\Phi_\eps(F)\setminus E_k\subseteq B(x_0, r)$ and then, recalling \eqref{erre},  $d_{E_{k-1}}(x)\geq h\delta-r>h\left(\delta-\frac{\delta\lambda}{2}\right) $ for all $x\in \Phi_\eps(F)\setminus E_k$. Therefore
  we get 
\[\left(\delta -\frac{\delta\lambda}{2}\right)|\Phi_\eps(F)\setminus E_k|< \left(\delta-\delta \lambda-\frac{\rho}{2}\right) |\Phi_\eps(F)\setminus E_k|\]   
which implies $\Phi_\eps(F)\subseteq E_k$, in contradiction with our construction. 
\end{proof} 

 We now prove the main result of this section, about the flow of $K$-outward minimizing sets.   
 
 \begin{theorem}\label{presmin}  
 Let $\Omega$ be an open set, $E$ a bounded set with  $E\subset\subset \Omega$.  Assume that $E$ is  a  strongly $K$-outward minimizing  set  in $\Omega$ with constant $\delta>0$.
  Then for every $t>0$ up to a countable set,  we have
 \[E_h(t) \to E^-(t) \qquad \text{ in $L^1(\Omega)$,\ as $h\to 0$. }\] 
  Moreover,
  $E^-(t)$ is a $K$-outward minimizing set in $\Omega$ for every $t>0$, and
\[ E^-(t+s)\subseteq  E^-(t)\ \text{ with }\ d(E^-(t+s), E^-(t))\geq \delta s \qquad \text{  for every  $t,s> 0$}. \]  Moreover $\cap_{s<t}E^-(s)\setminus E^-(t)$ has empty interior for all $t>0$, $|\cap_{s<t}E^-(s)\setminus E^-(t)|=0$ for every $t>0$ up to a countable set,  
  and 
 \[ H^K_{E^-(t)}(x)\geq \delta\qquad\text{for all $x\in  \partial E^-(t)$.} \]

Finally if  $E$ has boundary of class $C^{1,1}$, the same result holds also  for the outer flow $E^+(t)$ and   $E^{+}(t)\setminus E^-(t)$ has empty interior for all $t>0$.
\end{theorem} 
\begin{proof} 
Note that by Proposition \ref{prop2} and Remark \ref{ineout} we may assume $\text{int }E\not=\emptyset$, otherwise the statement is trivial. 

We divide the proof in several steps. 

\noindent {\bf Step 1: definition of a continuous minimal time function $u$}.

We recall that  $E_h (t) =T_h^{(k)}(E)$, for $t\in [kh, (k+1)h)$. 
We define the discrete  arrival time function as follows
\begin{equation}\label{arrivaldiscrete}
u_h(x)=\begin{cases} h\sum_{k\geq 0} \chi_{E_k}(x)=\int_0^{+\infty} \chi_{E_h(t)}dt & x\in E\\ 0 &x\in \R^n\setminus E.\end{cases} \end{equation} 
Note that by    Proposition \ref{prop2},  $u_h$ is well defined and \[ \left\{x\text{ s.t. }u_h(x)>t\right\}=E_h(t). \] 

By its very definition, we get that
\begin{equation}\label{th}
T_h^{(k)}u_h(x)= u_h(x)-hk\end{equation}
where $T_h^{(k)}u_h(x)$ is defined as in \eqref{thu}. 

  Moreover by    Proposition \ref{prop2}we get that $d(T_h^{(k)}(E),T_h^{(k')}(E))\geq \delta h(|k-k'|-1)$. 
Let $x\in T_h^{(k)}(E)$ and $y\in T_h^{(k')}(E)$, 
\[|u_h(x)-u_h(y)|= h|k'-k| \leq \frac{d(T_h^{(k)}(E), T_h^{(k')}(E))}{\delta}+h\leq \frac{|x-y|}{\delta}+h.\] This implies that up to a subsequence $u_h\to u$ uniformly as $h\to 0$, where $u:\R^n\to \R$ is a Lipschitz continuous function such that $u=0$ in $\R^n\setminus E$ and $|u(x)-u(y|\leq \frac{|x-y|}{\delta}$.

\vspace{0,1cm}

\noindent {\bf Step 2:  for all $t>0$ $E^-(t)=\{x\text{ s.t. } u(x)> t\} $}.

 Note that since $u_h\to u$ uniformly then it is also true that  $\|T_h u_h- T_h u\|_\infty\to 0$ as $h\to +\infty$ and then also 
 $\|T_h^{(k)}u-T_h^{(k)}u_h\|_\infty\to 0$ as $h\to 0$ for all $k\geq 1$, where $T_h u, T_h^{(k)}u$ are defined as in \eqref{thu}. 
 Therefore by Theorem \ref{convschema} we conclude that
 \[T_h^{\left[\frac{t}{h}\right]}u_h(x)\to u(x,t)\]   locally uniformly in $\R^n\times [0, +\infty)$  as $h\to 0$
 where $u(x,t)$ is the unique viscosity  solution to \eqref{levelset} with initial datum $u$. 

On the other side, by \eqref{th} we get that  \[T_h^{\left[\frac{t}{h}\right]}u_h(x)\to u(x)-t\] locally uniformly. 
This implies that $u(x)-t$ is the unique viscosity solution to \eqref{levelset} with initial datum $u$ and in particular, since the operator is geometric and the level set $\{u(x)>0\}$ coincide with the level set $\{d_E(x)>0\}$ we conclude that
\[E^-(t)= \{x\text{ s.t. }u(x)>t\} \qquad \forall t>0.\]
Note that by this equality we deduce also that \[\bigcap_{s<t}E^-(s)=\{x\text{ s.t. }u(x)\geq t\},\] and that the limit $u$ of $u_h$ is unique, so  the whole family  $u_h$ converges to $u$ uniformly as $h\to 0$. 
 By its characterization, we get also that $ E^-(t+s)\subseteq \{x\text{ s.t. }u(x)\geq t+s\}\subseteq  E^-(t)$  for all $s>0$. 
 
\vspace{0,1cm}

\noindent {\bf Step 3:  $L^1$ convergence and $K$-outward minimality property of $E^-(t)$}.

 By uniform convergence of $u_h\to u$, we get for all $t>0$ we have
\[E^-(t)=\{x\text{ s.t. }u(x)>t\} \subseteq \lim_{h\to 0}E_h(t)\subseteq \{x\text{ s.t. }u(x)\geq t\}=\bigcap_{s<t}E^-(s) \]   where the limit is taken in the $L^1$ sense. 
Since $u$ is Lipschitz continuous, we know that  $|\{x\text{ s.t. }u(x)=t\}|=0$ for almost every $t>0$, which implies that
$E_h(t)\to E^-(t)$  in $L^1(\R^n)$ for almost every $t>0$. 
Moreover by stability with respect to $L^1$ convergence of $K$-outward minimizing sets see Proposition \ref{conv}, since $E_h(t)$ are $K$-outward minimizers in $\Omega$ by Proposition \ref{prop1} we conclude that 
also $E^-(t)$, and  $\bigcap_{s<t}E^-(s)$ are  $K$-outward minimizer sets in $\Omega$ for almost every $t>0$.

Now we observe that  $E^-(t)$ is a $K$-outward minimizer set in $\Omega$ for every  $t>0$ again by stability under $L^1$ convergence, since  $E^-(t)=\cup_{s>0}E^-(t+s)=\lim_{s\to 0^+}(E^-(t+s))$. Then  also $\bigcap_{s<t}E^-(s)=\lim_{s\to t^-} E^-(s)$ is a $K$-outward minimizer set in $\Omega$ for every  $t>0$.
 
 \vspace{0,1cm}

\noindent {\bf Step 4: $K$-curvature   of $E^-(t)$.}

Since $E$ is strongly $K$-outward minimizer with $\delta>0$ then by Proposition \ref{prop2} we get that 
\[d(E_h(t), E_h(t+s))\geq \delta  \left(h\left[\frac{t+s}{h}\right]-h\left[\frac{t}{h}\right]-h\right)\geq s\delta -2h \delta.\]
 Then
 \[E_h(t+s)+B(0, \delta s-2h\delta)\subseteq E_h(t).\] Passing to the limit as $h\to 0$ we get that  for almost every $t,s>0$
\begin{equation}\label{mon} d(\{ u(x)\geq t+s\}, \{u(x)>t\})  \geq \delta s.\end{equation}   Arguing as before, we get that this inequality holds for all $s,t>0$. 

We apply now Theorem \ref{pos2}, choosing as initial set  $\{u(x)>t\}$ and observing that the outer flow at time $s>0$ of $\{u(x)>t\}$   is given by $\{u(x)\geq t+s\}$. So we get that   \begin{equation}\label{ccc}H^K_{\{u(y)>t\}}(x)\geq \delta\qquad \text{in viscosity sense for all $x\in \partial \{u(y)>t\}$ and for all $t>0$. }\end{equation} 
 
  \vspace{0,1cm}

\noindent {\bf Step 5: the set $\{x\text{ s.t. }u(x)=t\}$.}

We show  that for all $t>0$
\[\text{ int } \left( \{x\text{ s.t. }u(x)\geq t\}\setminus\{x\text{ s.t. }u(x)>t\}\right)=\text{ int }  \{x\text{ s.t. }u(x)=t\}=\emptyset.\]  We assume by contradiction that there exists $z$ and $r>0$ such that $B(z,r)\subset\subset  \{x\text{ s.t. }u(x)=t\}\subseteq  \{x\text{ s.t. }u(x)\geq t\}$. Let 
\[\alpha:=\max_{k\in [r/2, r]}\max_{y\in \partial B(0, k)} H^K_{B(0, k)}(y).\] Note that by definition of curvature, then $\alpha=\max_{y\in \partial B(0, r/2)} H^K_{B(0, r/2)}(y)>0$. Let $s_0>0$ such that $r-\alpha s_0>r/2$, and define the flow $B(s)= B(z,r-\alpha s)$ for $s\in [0, s_0]$. Then we get that $B(s)$ is a strict subsolution to \eqref{kflow} since 
$H^K_{B(s)}(y)\leq 2\alpha$ for every $s\in [0, s_0]$. Recalling that $u(x)-t$ is a viscosity solution to \eqref{levelset}, we conclude by Proposition \ref{subgeometrico}, that $B(z, r-\alpha s)=B(s)\subseteq  \{x\text{ s.t. }u(x)\geq t+s\}$. This implies that $t=u(z)\geq t+s$ for all $s\in [0, s_0]$ which is not possible. 

Moreover, observe that  by \eqref{mon}, the set of $t>0$ where $|\{x\text{ s.t. }u(x)=t\}|>0$ coincides with the set of jumps of the strictly decreasing function $t\to |E^-(t)|$.  Therefore, 
this set is countable. 

 \vspace{0,1cm}

\noindent {\bf Step 6: case of $E$ with $C^{1,1}$ boundary.}

Note that if $E$ is strongly $K$-outward minimizing, then by Theorem \ref{euler}, $H^K_E(x)\geq \delta$ for all $x\in\partial E$. Since $E$ has boundary of class $C^{1,1}$, then it is also a strongly $K$-mean convex set, see Remark \ref{regularset}. Therefore  by  Proposition \ref{chara}, item (1) we get that   $E^{+}(t)\subseteq E^-(t-s)$ for every $t>s>0$   and so 
\[E^{+}(t)\subseteq \bigcap_{0<s<t}E^-(t-s)=\bigcap_{0<s<t}\{x\text{ s.t }u(x)>t-s\}=\{x\text{ s.t }u(x)\geq t\}\subseteq E^+(t).\] This implies that for all $t>0$, $\{x\text{ s.t. }u(x)\geq t\}=E^+(t)$. 
\end{proof} 
 \begin{remark}\upshape  If  the outer flow satisfies  $E^+(t)\subset \subset  E$, for $t>0$,  then  the same results as in Theorem \ref{presmin} hold also for the outer flow $E^+(t)$, since we may prove that $\{x\text{ s.t. }u(x)\geq t\}=E^+(t)$, arguing exactly as in Step 5 of the proof. In particular we would  get that $E^+(t)\setminus E^-(t)$ has empty interior for all $t$. 
  
 We expect this monotonicity property to hold true for the flow starting by a strongly $K$-outward minimizer. 
 
\end{remark}

\begin{remark}\upshape In Theorem \ref{presmin} we show that the volume function
\[t\to |E^-(t)|\] is strictly decreasing. 
We expect that this function is also continuous, as it happens in the local case. 
 \end{remark} 
 
 We conclude with a corollary about the convergence of the $K$-perimeter of the discrete flow to the $K$-perimeter of the limit level set flow
 (we refer to \cites{dephilaux,chn} for analogous results in the local case).
  
 \begin{corollary}  Let $\Omega$ be a  domain and $E\subset\subset \Omega$   be a   strongly $K$-outward minimizing  set  in $\Omega$ with constant $\delta>0$.
  Then for every $T>0$ \[\int_0^T \Per_K(E_h(t) )dt \to \int_0^T \Per_K(E^-(t))dt  \qquad \text{  as $h\to 0$ }\] where $E_h(t)$ is any piecewise constant flow defined as in \eqref{pol} and $E^-(t)$ is the viscosity inner  flow as defined in \eqref{outin}.
 \end{corollary}
 
 \begin{proof} 
 By Theorem \ref{presmin}, $E_h(t)\to E^-(t)$ in $L^1(\Omega)$ for almost every $t$, therefore by lower semicontinuity of $\Per_K$ with respect to $L^1$ convergence and Fatou lemma, we get that for every $T>0$, 
 \begin{equation}\label{lsc}
 \liminf_{h\to 0} \int_0^T \Per_K(E_h(t) )dt\geq  \int_0^T \Per_K(E^-(t))dt.\end{equation} 
 
 We now introduce the functional 
   \begin{equation}\label{funz1}
J_K(v):=\frac 12 \int_{\R^n}\!\!\int_{\R^n} |v(x)-v(y)|K(x-y)dxdy \qquad v \in L^1_{loc}(\R^n).
\end{equation} 
Note that $J_K(\chi_E)=\Per_K(E)$ for all measurable $E\subset \R^n$.
The coarea formula \cite[Proposition 2.3]{cn} states that  
\begin{equation}\label{coarea} J_K(v)= \int_{-\infty}^{+\infty} \Per_K(\{v>s\})ds\end{equation}
 for  all $v \in L^1_{loc}(\R^n)$.
 
Let $u_h$ as defined in \eqref{arrivaldiscrete} and we claim that
\begin{equation}\label{1var} J_K(u_h)\leq J_K(v) \qquad \text{ for all $v\in L^1_{loc}(\R^n)$, $v\geq u_h$ and $\text{supp } v\subset\subset  \Omega$.}\end{equation} 
The proof of this claim is a direct consequence of the coarea formula and the fact that $E_h(t)$ is $K$-outward minimizer for every $t$, by  Proposition \ref{prop1}. Indeed, since $u_h\leq v$, there holds for every $s>0$ that 
\[E_h(s)=\{x\text{ s.t. }u_h(x)>s\}\subseteq \{x\text{ s.t. }v(x)>s\}\subset\subset \Omega\] which implies, since $E_h(t)$ is a $K$-outward minimizing set, that 
\[\Per_K(E_h(s))=\Per_K(\{x\text{ s.t. }u_h>s\})\leq \Per_K( \{x\text{ s.t. }v>s\}).\]  Integrating for $s\in (0, +\infty)$, and recalling \eqref{coarea}, we get the conclusion. 

Now, we use the same argument as in \cite[Proposition 5.1]{dephilaux}. We recall that by Theorem \ref{presmin}, $u_h\to u$ uniformly as $h\to 0$, where $u$ is Lipschitz continuous and $u=0$ in $\R^n\setminus E$. By uniform convergence we get that  for any $\eps>0$ there exists $h_0$ such that $u_h\leq u+\eps$ for all $h<h_0$.  
Let   $v(x):= (u(x)+\eps)\chi_Ei(x)$, so $v(x)\geq u_h(x)$ by construction and moreover $\text{supp }v=E\subset\subset  \Omega$. 

Therefore by \eqref{1var} there holds that
\begin{eqnarray*} J_K(u_h) &\leq &J_K(v)=J_K((u+\eps)\chi_E)\\ &\leq & J_K(u)+  J_K(\eps\chi_E)= J_K(u)+\eps\Per_K(E).  \end{eqnarray*} 
Sending  $\eps \to 0$  we  conclude that 
\begin{equation} \label{prova} J_K(u_h) \leq J_K(u)\end{equation} 
Recalling that $E_h(t)=\{x\text{ s.t. } u_h(x)>t\}$ and $E^-(t)=\{x\text{ s.t. } u(x)>t\}$, \eqref{prova}, by the coarea formula, coincides with
\[\int_0^{+\infty} \Per_K(E_h(t))dt\leq \int_0^{+\infty}  \Per_K(E^-(t))dt\qquad \text{for all $h\leq h_0$.} \]
This inequality, together with \eqref{lsc}, gives the thesis.
 \end{proof} 
 

\end{document}